\documentclass[10pt,reqno]{amsart}
\usepackage{amsmath,amscd,amsthm,amsfonts,amsopn,amssymb,mathrsfs}
\newtheorem{theorem}{Theorem}[section]
\newtheorem{proposition}[theorem]{Proposition}

\newtheorem{definition}[theorem]{Definition}
\newtheorem{corollary}[theorem]{Corollary}
\newtheorem{remark}[theorem]{Remark}
\newtheorem{lemma}[theorem]{Lemma}

\allowdisplaybreaks
\begin{document}

\title{Characterization of isoperimetric sets \\
inside almost-convex cones}
\author{Eric Baer}
\address{Department of Mathematics, University of Wisconsin--Madison, 480 Lincoln Dr., Madison, WI 53706}
\email{ebaer@math.wisc.edu}
\author{Alessio Figalli}
\address{The University of Texas at Austin, Mathematics Department, 2515 Speedway Stop C1200, Austin, TX 78712, USA}
\email{figalli@math.utexas.edu}

\thanks{Date: April 29, 2016.\\
The work of E.B. was partially supported by the National Science Foundation under Award Nos. DMS-1204557 and DMS-1147523.
The work of A.F. was partially supported by the National Science Foundation under Award Nos. DMS-1262411 and DMS-1361122.}

\begin{abstract}
In this note we characterize isoperimetric regions inside almost-convex cones.
More precisely, as in the case of convex cones, we show that 
isoperimetric sets are given by intersecting the cone with a ball centered at the origin. \end{abstract}

\maketitle

\section{Introduction}

Fix $N\geq 3$, and for each set of directions $A\subset S^{N-1}$ let
\begin{align*}
\mathscr{C}_A:=\bigcup_{t>0}\{t\xi:\xi\in A\}
\end{align*}
be the cone associated to $A$.  In this note we characterize isoperimetric sets in $\mathscr{C}_A$ (that is, sets which minimize the relative perimeter while satisfying a volume constraint), for cones which are asymptotically close to a certain class of fixed convex cones.  The characterization we obtain is that the isoperimetric set is given by the intersection of the cone with a ball centered at the origin, having radius chosen to satisfy the volume constraint.

When the cone is assumed convex, this type of characterization has a rich history.  The first result in this direction was obtained by P.L. Lions and F. Pacella \cite{LP} using the Brunn--Minkowski inequality.  For related studies of isoperimetric sets in cones, we refer to \cite{Morgan,RR,FI,CROS,CR,RV}, and the references cited in these works.

In the present work, we are motivated by recent progress in stability estimates which allow the characterization of isoperimetric sets in the convex cone to be extended into the non-convex range, at least for small perturbations. In particular, we shall rely on the approach introduced by Fuglede in \cite{F} 
(and further developed in many other papers, see for instance \cite{CL,FM,FJ,FFMMM,CLM} and the references therein) to establish a quantitative stability estimate for the isoperimetric inequality on sets which are small perturbations of the sphere. 
However, the result in \cite{F} exhibits a strong dependence on the translation invariance present in the Euclidean isoperimetric problem.  Since this invariance is no longer present when working in a cone strictly contained in $\mathbb{R}^N$, we must pursue an alternative approach.  For this, we restrict to cones contained in a spherical cap strictly smaller than the hemisphere.  The desired bounds then follow from estimates on stability of eigenvalues with Neumann boundary conditions.

To precisely state our main result, we need some additional notation.  For an open set $F\subset\mathbb{R}^N$ and a set of finite perimeter $E\subset\mathbb{R}^N$, let $P(E;F)$ denote the perimeter of $E$ relative to $F$; we refer to \cite{AFP} and \cite{Maggi} for definitions and fundamental properties of sets of finite perimeter and the relative perimeter functional.

Moreover, let $\pi_N:\mathbb{R}^{N}\rightarrow\mathbb{R}$ be the projection given by $\pi_N(x_1,\cdots,x_N):=x_N$ for $x=(x_1,\cdots,x_N)\in\mathbb{R}^N$.  For each $\eta>0$, let $S_+(\eta)\subset S^{N-1}$ be the spherical cap given by
\begin{align}
S_+(\eta):=\{\xi\in S^{N-1}:\pi_N(\xi)>\eta\}.\label{def-splus}
\end{align}

We now introduce the notion of uniform $C^{1,1}$ domains in the $S_+(\eta)$ setting.
This is defined requiring that every point of the boundary has both an interior and an exterior ball condition.

\begin{definition}[Uniform $C^{1,1}$ class $\Pi_+(\eta,r)$]
\label{def-unif-lip}
Given $\eta>0$ and $r>0$, we let $\Pi_+(\eta,r)$ denote the class of open sets $A$ compactly contained in $S_+(\eta)$ such that for every $x\in \partial A$ there exist a ball $B_r^+\subset S^{N-1}$ contained in $A$ and a ball $B_r^-\subset S^{N-1}$ contained in $S^{N-1}\setminus A$, both of radius $r$, such that $x \in \partial B_r^+\cap \partial B_r^-$.
\end{definition}

We are now ready to state our main result.  In this statement, as well as in the remainder of the paper, for two sets $X,Y\subset S^{N-1}$, the quantity $d_{L^\infty}(\partial X,\partial Y)$ refers to the Hausdorff distance between $\partial X$ and $\partial Y$ defined via the intrinsic metric on the sphere.

\begin{theorem}
\label{clm11}
Fix $N\geq 3$ and $\eta>0$.  Then, for any $r>0$, there exists $\epsilon=\epsilon(N,\eta,r)$ such that the following holds: given $A'\in \Pi_+(\eta,r)$, assume that there exists $A\in \Pi_+(\eta,r)$ such that 
\begin{align*}
\mathscr{C}_A\quad\textrm{is convex and}\quad d_{L^\infty}(\partial A',\partial A)<\epsilon.
\end{align*}
Then, for all $m>0$, the unique minimizer $E'_*$ of the functional $E\mapsto P(E;\mathscr{C}_{A'})$ among subsets $E'$ of $\mathscr{C}_{A'}$ having finite perimeter and $|E'|=m$ is given by $E'_*=B\cap \mathscr{C}_{A'}$, where $B\subset\mathbb{R}^N$ is the ball centered at the origin with radius chosen so that $|B\cap \mathscr{C}_{A'}|=m$.
\end{theorem}

\begin{remark}
{\em 
By the analysis performed in this paper (see in particular Lemma \ref{lem1}) it follows that $B\cap \mathscr{C}_{A'}$ is a uniformly stable minimizer (in other words,
the second variation of the perimeter functional at fixed volume is uniformly positive).
Thanks to this fact, it follows by our analysis and the selection principle introduced in \cite{CL} that a sharp quantitative version of the isoperimetric inequality
as the one in \cite{FI} (see also \cite{FuMP,FMP,CL,FJ}) holds in our setting.}
\end{remark}

\begin{remark}
{\em 
Since our main Theorem \ref{clm11} above is stated for $C^{1,1}$ domains,
we will also assume in all other statements that our sets have $C^{1,1}$ boundary.
However, it is interesting to observe that 
most of our arguments require only the Lipschitz regularity of our sets, rather than their $C^{1,1}$ regularity. The crucial point where we need higher regularity on our sets is to show existence of minimizers (see Proposition \ref{prop-existence}) and to apply the $\epsilon$-regularity theory for almost minimizers of the perimeter up to the boundary (see for instance \cite{DM} and the references therein). Although one could probably slightly weaken this assumption to some $C^{1,\alpha}$ regularity with $\alpha>0$, it is unclear to us whether our method can be extended to Lipschitz domains.}
\end{remark}

\medskip
{\it Acknowledgments:} The intuition that the result of Lions and Pacella \cite{LP} could be extended to almost-convex cones was already suggested by Ennio De Giorgi. The second author is grateful to Filomena Pacella for pointing out this problem to his attention.
Both authors thank Gian Paolo Leonardi for a careful reading of the manuscript and for pointing out a small gap in one of the proofs.

\section{Preliminaries}
\label{sect:prelim}

In the next three subsections, we prepare several analytic results which will be used in the proof of Theorem $\ref{clm11}$.  In Section $2.1$ we discuss the existence of minimizers stated in the theorem.  In Section $2.2$ we adapt arguments of Chenais \cite{Chenais} from the Euclidean setting to show the existence of uniformly bounded extension operators for sets in the class $\Pi_+(\eta,r)$ (with suitable values of the parameters); as a consequence, we obtain some stability results for eigenvalues of the Neumann Laplacian associated to domains in this class.  Lastly, in Section 2.3 we use the stability results obtained in Section 2.2 to establish a sharp Poincar\'e estimate which will form the core analytic ingredient in the proof of Theorem $\ref{clm11}$.

\subsection{Existence of minimizers}

In this section, we apply the result of Ritor\'e and Rosales \cite{RR} to establish existence of minimizers in the setting of Theorem $\ref{clm11}$.  We in particular appeal to \cite[Proposition $3.5$]{RR}, which gives existence of minimizers for smooth ($C^{1}$ suffices) cones in $S_+$
(note that the existence of minimizers is known also for non-smooth convex cones, and it has been proved in \cite{FI} using optimal transport techniques).

Hence, we can state the following result:

\begin{proposition}[Existence of minimizers in the setting of Theorem \ref{clm11}]
\label{prop-existence}
Fix $\eta>0$ and $r>0$.  Then for every $m>0$ and $A'\in \Pi_+(\eta,r)$ there exists a set of finite perimeter $E_*'\subset \mathscr{C}_{A'}$ having $|E_*'|=m$ and such that $E_*'$ is a minimizer for the functional $E\mapsto P(E;\mathscr{C}_{A'})$ among all such sets.
\end{proposition}

\subsection{Extension results and stability of eigenvalues}

In what follows, we let $W^{1,2}(S^{N-1})$ be the usual Sobolev space defined on the sphere (see, e.g. \cite[\S 7.3]{LM}).  Moreover, for each open set $A\subset S^{N-1}$, let $W^{1,2}(A)$ be the space obtained by completion of smooth functions with respect to the norm
\begin{align*}
\lVert u\rVert_{W^{1,2}(A)}^2:=\lVert u\rVert_{L^2(A)}^2+\lVert \nabla u\rVert_{L^2(A)}^2,
\end{align*}
where $\nabla$ is the usual gradient operator on the sphere.

Fix $\eta>0$.  Then for $r$ as in the definition of the uniform $C^{1,1}$ class (Definition $\ref{def-unif-lip}$), restricting considerations to domains in $\Pi_+(\eta,r)$ allows us to note that the analysis of \cite{Chenais} in the Euclidean setting  implies the existence of a family of uniformly bounded extension operators.

\begin{proposition}[Domains in $\Pi_+$ have uniform extension operators] 
\label{prop-extension}
Fix $N\geq 3$, $\eta>0$, and $r>0$.  Then there exists $C>0$ such that for each $A\in \Pi_+(\eta,r)$ there exists a linear extension operator
\begin{align*}
T_A:W^{1,2}(A)\rightarrow W^{1,2}(S_+)
\end{align*}
with
\begin{align*}
\lVert T_A\varphi\rVert_{W^{1,2}(S_+)}\leq C\lVert \varphi\rVert_{W^{1,2}(A)},
\end{align*}
where $S_+=S_+(\eta)$ is defined as in ($\ref{def-splus}$).
\end{proposition}

\begin{proof}
Since $A\subset S_+\subset \subset S^{N-1}\cap \{x_N>0\}$,
we note that the projection
$$
\hat \pi_N:\mathbb{R}^N\to \mathbb{R}^{N-1},\qquad \hat\pi_N(x',x_N):=x',
$$
is a smooth diffeomorphism from $S_+$ to $\hat\pi_N(S_+)$.
Hence, it is enough to apply \cite[Theorem II.1]{Chenais} 
to find a linear bounded extension operator from $W^{1,2}(\hat\pi_N(A))$ to $W^{1,2}(\hat\pi_N(S_+))$,
and then lift the result back to the sphere using the smooth map $\hat\pi_N^{-1}:\hat\pi_N(S_+)\to S_+$.
\end{proof}

In the rest of this subsection, we apply Proposition $\ref{prop-extension}$ to obtain stability results for the first non-zero Neumann eigenvalue of $-\Delta$ on domains $A$ in the class $\Pi_+=\Pi_+(\eta,r)$ with suitable values of $\eta$ and $r$.  

More precisely, for $A\in \Pi_+$ we set
\begin{align}
\mu_1(A):=\inf\bigg\{\frac{\lVert \nabla u\rVert_{L^2(A)}}{\lVert u\rVert_{L^2(A)}}:u\in W^{1,2}(A),\int_A u\,d\mathcal{H}^{N-1}=0, u\not\equiv 0\bigg\}.\label{eq-variational}
\end{align}
(Here and in the following, $d\mathcal{H}^{N-1}$ denotes the surface measure on $S^{N-1}$.)

Our first result in this direction is a pointwise stability result for $\mu_1$. 

\begin{proposition}[Pointwise stability of $\mu_1$ in $\Pi_+$]
\label{prop-convergence-eigenvalues}
Fix $N\geq 3$, $\eta>0$, and $r>0$.  If $(A_n)_{n\geq 1}$ is a sequence of sets in $\Pi_+=\Pi_+(\eta,r)$ such that 
\begin{align*}
d_{L^\infty}(\partial A_n,\partial A)\rightarrow 0\,\,\textrm{as}\,\, n\rightarrow\infty
\end{align*}
for some $A\in \Pi_+$, then $\mu_1(A_n)\rightarrow \mu_1(A)\,\,\textrm{as}\,\, n\rightarrow\infty$.
\end{proposition}

Since our domains are uniformly $C^{1,1}$, the stability of $\mu_1$ is actually easy to prove using that $A_n$ and $A$ can be mapped one into the other via a bi-Lipschitz map with bi-Lipschitz constant converging to $1$ as $n\to \infty$. However, it is interesting to notice that 
Proposition $\ref{prop-convergence-eigenvalues}$
holds also for Lipschitz domains, and the proof in this case is essentially due to Chenais \cite{Chenais} in the Euclidean case; see also the treatments in \cite[Theorem 2.3.25]{Henrot} and \cite[Theorem 7.4.25]{BucurButtazzo}, and the references cited in these works, as well as recent work \cite{DCL} in the general setting of metric measure spaces.
 
To emphasize the fact that Proposition \ref{prop-convergence-eigenvalues} holds for Lipschitz domains, we present a detailed outline of Chenais' argument for the convenience of the reader.  Our discussion also follows the arguments leading to the proof of \cite[Corollary $7.4.25$]{BucurButtazzo}.

\begin{proof}
For each $n\geq 1$, let $T_n:W^{1,2}(A_n)\rightarrow W^{1,2}(S_+)$, $n\geq 1$, and $T:W^{1,2}(A)\rightarrow W^{1,2}(S_+)$ be the extension operators given by Proposition $\ref{prop-extension}$, and let $P_n:L^2(S_+)\rightarrow L^2(A_n)$, $n\geq 1$, and $P:L^2(S_+)\rightarrow L^2(A)$ denote the restriction operators given by $P_nf:=f|_{A_n}$ and $Pf:=f|_A$ for $f\in L^2(S_+)$.  

Now, define 
\begin{align*}
R_n:=(I-\Delta_{A_n}^{(N)})^{-1}\circ P_n:L^2(S_+)\rightarrow W^{2,2}(A_n)
\end{align*}
for $n\geq 1$, and, similarly,
\begin{align*}
R:=(I-\Delta_A^{(N)})^{-1}\circ P:L^2(S_+)\rightarrow W^{2,2}(A)
\end{align*}
as the resolvent operators associated to the Neumann Laplacian on $A_n$, $n\geq 1$, and $A$, respectively.  In what follows, we will use the convention that for all $f\in L^2(S_+)$, $R_nf$ and $Rf$ are extended by $0$ outside $A_n$ and $A$, respectively.

We first note that standard elliptic PDE estimates on $S_+$ lead to the definition of $R_n$, $n\geq 1$, and $R$ as compact, positive-definite, self-adjoint operators on the appropriate subspaces $L^2(A_n)$ and $L^2(A)$.  Moreover, we have the bounds
\begin{align}
\lVert R_nf\rVert_{W^{2,2}(A_n)}\lesssim \lVert f\rVert_{L^2(A_n)},\,\,n\geq 1,\quad\textrm{and}\quad \lVert Rf\rVert_{W^{2,2}(A_n)}\lesssim \lVert f\rVert_{L^2(A)}.\label{eq-elliptic}
\end{align}

Let $\lambda_n$, $n\geq 1$, and $\lambda_*$ denote the second eigenvalues (when arranged in decreasing order) of $R_n$, $n\geq 1$, and $R$, respectively.  Then 
\begin{align*}
\mu_1(A_n)=\sqrt{(1/\lambda_n)-1},\,\,n\geq 1
\end{align*}
and
\begin{align*}
\mu_1(A)=\sqrt{(1/\lambda_*)-1}.
\end{align*}
To show the desired convergence of $\mu_1(A_n)$ to $\mu_1(A)$, it therefore suffices to show convergence of $\lambda_n$ to $\lambda_*$.  For this, standard functional analysis arguments (see, e.g. \cite[Theorem $2.3.1$]{Henrot}) reduce the issue to showing
\begin{align}
\lVert R_n- R\rVert_{L^2(S_+)\rightarrow L^2(S_+)}\rightarrow 0\label{eq-star}
\end{align}
as $n\rightarrow \infty$.

Following \cite{BucurButtazzo}, by the definition of the operator norm, choose a sequence $(f_n)\subset L^2(S_+)$ with $\lVert f_n\rVert_{L^2(S_+)}=1$ so that for each $n\geq 1$,
\begin{align*}
\lVert R_n-R\rVert_{L^2(S_+)\rightarrow L^2(S_+)}&\leq \lVert (R_n-R)f_n\rVert_{L^2(S_+)}+\frac1n.
\end{align*}
Since the sequence $(f_n)$ is bounded in $L^2(S_+)$, we may find a subsequence which converges weakly to some $f\in L^2(S_+)$.  Renumbering the subsequence as $(f_n)$ again, this leads to
\begin{align}
\nonumber \lVert R_n-R\rVert_{L^2(S_+)\rightarrow L^2(S_+)}&\leq\limsup_{n\rightarrow\infty}\lVert R_n(f_n-f)\rVert_{L^2(S_+)}\\
\nonumber &\hspace{0.2in}+\limsup_{n\rightarrow\infty} \lVert (R_n-R)f\rVert_{L^2(S_+)}\\
&\hspace{0.2in}+\limsup_{n\rightarrow\infty} \lVert R(f-f_n)\rVert_{L^2(S_+)}.\label{eq-three}
\end{align}

We begin by observing that
\begin{align}
\lVert R_nf-Rf\rVert_{L^2(S_+)}\rightarrow 0\quad\textrm{as}\,\,n\rightarrow\infty.\label{eq-rn-r}
\end{align}
For this, we refer to the proofs of \cite[Theorem 7.2.7 and Proposition 7.2.4]{BucurButtazzo}, which hold equally well in the $S_+$ setting (the arguments are expressed in the language of Mosco convergence, but ultimately rely on basic properties of the operators $R_n$ and $R$, and the existence of the extension operators $T_n$ and $T$).

We now estimate the remaining terms.  Observe that for each $n\geq 1$, the estimates ($\ref{eq-elliptic}$) combined with the boundedness of $T_n$ give
\begin{align*}
\lVert T_nR_n(f_n-f)\rVert_{W^{1,2}(S_+)}&\leq C\lVert R_n(f_n-f)\rVert_{W^{1,2}(A_n)}\leq 2C\lVert f\rVert_{L^2}
\end{align*}
so that the compactness of the embedding $W^{1,2}(S_+)\hookrightarrow L^2(S_+)$ implies we may find $g\in L^2(S_+)$ such that $T_nR_n(f_n-f)$ converges strongly to $g$ in $L^2(S_+)$, up to extraction of a subsequence.  This implies that $\|R_n(f_n-f)-g\|_{L^2(A_n)}\to 0$.  On the other hand, for any $\phi\in L^2(S_+)$, we have
\begin{align}
\nonumber \langle R_n(f_n-f),\phi\rangle_{L^2(A_n)}&=\langle f_n-f,R_n\phi\rangle\\
&=\langle f_n-f,(R_n-R)\phi\rangle+\langle f_n-f,R\phi\rangle.\label{eq-rn-phi}
\end{align}
Since $R_n\phi$ converges strongly to $R\phi$ in $L^2(S_+)$ (see ($\ref{eq-rn-r}$) above, which did not rely on any properties of $f$ besides $f\in L^2(S_+)$), the right-hand side of (\ref{eq-rn-phi}) tends to $0$ as $n\rightarrow\infty$ (in particular, we note that the strong convegence also holds in $L^2(A_n)$, and use the weak convergence of $f_n$ to $f$ on each term).  This means that $R_n(f_n-f)$ converges weakly to $0$ in $L^2(A_n)$, so that uniqueness of weak limits gives $g|_{A}=0$.  Thus, $R_n(f_n-f)$ converges strongly to $0$ in $L^2(A_n)$ (and therefore in $L^2(S_+)$ as well).
Also, an identical argument with $R_n$ replaced by $R$ and $T_n$ replaced by $T$ shows that $R(f_n-f)$ converges strongly to $0$ in $L^2(A)$ (and thus in $L^2(S_+)$).

Assembling the above estimates with ($\ref{eq-three}$), we obtain ($\ref{eq-star}$) as desired.
\end{proof}

A compactness argument now gives uniform continuity of the eigenvalue.

\begin{corollary}[Uniform continuity of $\mu_1$ in $\Pi_+$]
\label{cor-unif-eigenvalues}
Fix $N\geq 3$, $\eta>0$, and $r>0$.  Then for all $\epsilon>0$ there exists $\delta>0$ such that for all $A,A'\in \Pi_+(\eta,r)$,
\begin{align*}
d_{L^\infty}(\partial A',\partial A)<\delta
\end{align*}
implies 
\begin{align*}
|\mu_1(A')-\mu_1(A)|<\epsilon.
\end{align*}
\end{corollary}

\begin{proof}
Suppose the conclusion fails.  Then for some $\epsilon>0$ we can find sequences $(A_n)$, $(A_n')\subset \Pi_+(\eta,r)$ such that $d_{L^\infty}(\partial A_n', \partial A_n)\rightarrow 0$ as $n\rightarrow\infty$ and $|\mu_1(A_n')-\mu_1(A_n)|\geq \epsilon$.  Since $\Pi_+(\eta,r)$ is compact with respect to the topology induced by $d_{L^\infty}(\cdot,\cdot)$, we can find a sequence $(n_k)$ and a set $A_*\in \Pi_+(\eta,r)$ such that $d_{L^\infty}(\partial A_{n_k},\partial A_*)\rightarrow 0$ as $k\rightarrow\infty$.  But this implies that for $k\geq 1$ sufficiently large, one has
\begin{align*}
|\mu_1(A_{n_k})-\mu_1(A_*)|<\epsilon/4.
\end{align*}
On the other hand, since
\begin{align*}
d_{L^\infty}(\partial A_{n_k}',\partial A_*)\leq d_{L^\infty}(\partial A_{n_k}',\partial A_{n_k})+d_{L^\infty}(\partial A_{n_k},\partial A_*),
\end{align*}
with the right-hand side approaching $0$ as $k\rightarrow\infty$, we also have
\begin{align*}
|\mu_1(A_{n_k}')-\mu_1(A_*)|<\epsilon/4
\end{align*}
for $k$ sufficiently large.  This leads to
\begin{align*}
|\mu_1(A_{n_k})-\mu_1(A_{n_k}')|\leq \frac{\epsilon}{2},
\end{align*}
which gives the desired contradiction.
\end{proof}

\subsection{A refined Poincar\'e inequality}

We now give a refined Poincar\'e inequality for smooth perturbations of the convex cone $\mathscr{C}_A$.  We point out that our proof of this estimate exhibits a strong dependence on the assumption $\eta>0$ on the spherical cap $S_+=S_+(\eta)$ (that is, that $S_+$ is a spherical cap strictly smaller than the hemisphere).  
\begin{proposition}
\label{prop-neumann-evalue}
Fix $N\geq 3$, $\eta>0$, and $r>0$.  Let $\Pi_+:=\Pi_+(\eta,r)\subset S_+$ be as in Definition $\ref{def-unif-lip}$, with $S_+=S_+(\eta)$ as in (\ref{def-splus}).  Then there exist $\epsilon=\epsilon(N,\eta,r)$ and $c_1=c_1(\eta)$ satisfying
\begin{align*}
\epsilon>0\quad\textrm{and}\quad c_1(\eta)>\sqrt{N-1}
\end{align*}
such that the following holds: let
$A'\in \Pi_+(\eta,r)$, and assume that there exists $A\in \Pi_+(\eta,r)$ such that 
\begin{align*}
\mathscr{C}_A\quad\textrm{is convex and}\quad d_{L^\infty}(\partial A',\partial A)<\epsilon.
\end{align*}
Then
\begin{align*}
\mu_1(A')>c_1(\eta).
\end{align*}
\end{proposition}

\begin{proof}
Set $\Pi_+^{\textrm{conv}}:=\{A\in \Pi_+(\eta,r):\mathscr{C}_A\,\,\textrm{is convex}\}$.  We claim that there exists $c_0(\eta)>\sqrt{N-1}$ such that
\begin{align*}
\mu_1(A)\geq c_0(\eta)
\end{align*}
for all $A\in\Pi_+^{\textrm{conv}}$.  Indeed, suppose that this fails.  Then there exists a sequence $(A_n)$ in $\Pi_+^{\textrm{conv}}$ with $\limsup_{n\rightarrow\infty}\mu_1(A_n)\leq \sqrt{N-1}$.  By standard compactness results, we can find a set $A_*\in \Pi_+^{\textrm{conv}}$ and a subsequence $(A_{n_k})$ of $(A_n)$ with $\partial A_{n_k}$ converging to $\partial A_*$ in $d_{L^\infty}$.  By Proposition \ref{prop-convergence-eigenvalues}, one has
\begin{align*}
\mu_1(A_*)=\lim_{n\rightarrow\infty} \mu_1(A_{n_k})\leq \sqrt{N-1}
\end{align*}
On the other hand, \cite[Theorem $4.3$]{Escobar} (see also \cite[Section $3$]{Mazya1} and \cite[Theorem $4.1$]{Mazya2}) gives the lower bound $\mu_1(A_*)\geq \sqrt{N-1}$.  Moreover, 
since $A_*\in S_+(\eta)$ for $\eta>0$, it follows by \cite[Theorem $4.3$]{Escobar} 
or \cite[Theorem $4.1$(iv)]{Mazya2} that $\mu_1(A_*)>\sqrt{N-1}$, a contradiction that establishes the claim.

We now make the appropriate choice of $c_1(\eta)$, setting
\begin{align*}
c_1(\eta):=\frac{c_0(\eta)+\sqrt{N-1}}{2}
\end{align*}
Since $c_0(\eta)>\sqrt{N-1}$ implies $c_1(\eta)<c_0(\eta)$, the condition $d_{L^\infty}(\partial A, \partial A')<\epsilon$ yields (see Corollary \ref{cor-unif-eigenvalues})
\begin{align*}
\mu_1(A')>\mu_1(A)-(c_0(\eta)-c_1(\eta))\geq c_1(\eta),
\end{align*}
provided $\epsilon$ is sufficiently small.
\end{proof}

\begin{corollary}
\label{cor-poincare}
Fix $N\geq 3$, $\eta>0$, and $r>0$.  Then, for the values of $\epsilon$ and $c_1$ given by Proposition $\ref{prop-neumann-evalue}$, and any $A,A'\in \Pi_+(\eta,r)$ satisfying the conditions 
\begin{align*}
\mathscr{C}_A\quad\textrm{convex and}\quad d_{L^\infty}(\partial A',\partial A)<\epsilon,
\end{align*}
every $u\in W^{1,2}(A')$ satisfies
\begin{align}
\lVert u-\overline{u}\rVert_{L^2(A')}\leq \frac{1}{c_1(\eta)}\lVert \nabla u\rVert_{L^2(A')}\label{eq-poincare}
\end{align}
with $\overline{u}=\frac{1}{\mathcal{H}^{N-1}(A')}\int_{A'} u\,d\mathcal{H}^{N-1}$.
\end{corollary}

The proof of Corollary \ref{cor-poincare} follows from the classical observation 
that the variational characterization \eqref{eq-variational} of $\mu_1$ leads to identification of the sharp constant for Poincare inequalities of the form given in ($\ref{eq-poincare}$); for further details, see \cite[pp. 925--927]{DautrayLions}, as well as \cite{BC} and \cite{Ruiz} in which related uniform Poincar\'e estimates are obtained.

\begin{proof}
We apply (\ref{eq-variational}) to estimate norms of $u-\overline{u}$.  We can assume that $u-\overline{u}$ is not identically $0$ on $A'$.  Then, since 
\begin{align*}
\int_{A'} (u-\overline{u})\,d\mathcal{H}^{N-1}=\int_{A'} u\,d\mathcal{H}^{N-1}-\overline{u}\int_{A'}d\mathcal{H}^{N-1}=0,
\end{align*}
we obtain, from (\ref{eq-variational}),
\begin{align*}
\lVert u-\overline{u}\rVert_{L^2(A')}\leq \frac{1}{\mu_1(A')}\bigg(\int_{A'} |\nabla u|^2d\mathcal{H}^{N-1}\bigg)^{1/2},
\end{align*}
where we used that $\nabla (u-\overline{u})=\nabla u$ on $A'$.  The desired conclusion now follows from the bound $1/\mu_1(A')<1/c_1(\eta)$ ensured by Proposition $\ref{prop-neumann-evalue}$.
\end{proof}

\begin{remark}{\rm
An alternative argument, closer to the spirit of \cite{F}, can be given based on an eigenfunction expansion of $u$ with respect to the Laplace-Beltrami operator with vanishing Neumann boundary conditions.  For a description of this approach, see the outline for the sharp Poincar\'e inequality in the sphere and half-sphere given by Exercise 21 on pg. 538 of \cite{D}.}
\end{remark}

\section{Proof of Theorem $\ref{clm11}$}

In this section, we give the proof of our main result, Theorem $\ref{clm11}$, which gives the characterization of isoperimetric sets as sectors of balls of appropriate radius to match the volume constraint.

We begin with a preliminary lemma giving a lower bound for the relative perimeter of ``nearly sector''-type sets.  Let  $A\in \Pi_+(\eta,r)$, and let $E$ be a star-shaped subset of $\mathscr{C}_A$ which intersects every ray of 
$\mathscr{C}_A$.
For each $\xi\in A$ we define 
\begin{align*}
u(\xi):=\sup\, \Big\{t>-1:(1+t)\xi\in E\Big\},
\end{align*}
so that 
$$
E=\{t \xi: 0<t<u(\xi),\, \xi \in A\},
$$
and
$\partial E\cap \mathscr{C}_A$ can be parametrized as
\begin{align}
(1+u(\xi))\xi,\quad \xi\in A.\label{eq-un}
\end{align}
This yields the integral formula
\begin{align}
P(E;\mathscr{C}_{A})&=\int_{A} (1+u(\xi))^{N-1}\sqrt{1+\frac{|\nabla u(\xi)|^2}{(1+u(\xi))^2}}\,d\mathcal{H}^{N-1}(\xi).\label{eq-area}
\end{align}

As in \cite{F}, the argument is based on an analysis of leading order terms for this area integrand.  Accordingly, we will need several elementary inequalities, essentially identical to those used for the same purpose in \cite{F}.  We collect these here for the convenience of the reader.  They are
\begin{align}
(1+t)^k\geq 1-k|t|,\quad t\in\mathbb{R}, k\in\mathbb{N},\label{eq-tk}
\end{align}
and
\begin{align}
1+\tfrac{t}{2}-\tfrac{t^2}{8}\leq \sqrt{1+t}\leq 1+\tfrac{t}{2},\quad t\geq 0,\label{eq-sqrt}
\end{align}
together with
\begin{align}
\nonumber 1-\frac{N-1}{2}t^2-C|t|^3&\leq (1+t)^{N-1}-\frac{N-1}{N}\Big((1+t)^N-1\Big)\\
&\leq 1-\frac{N-1}{2}t^2+C|t|^3,\label{eq-poly1}
\end{align}
and
\begin{align}
|(1+t)^N-1-Nt|\leq Ct^2,\label{eq-poly2}
\end{align}
each holding for $t\in\mathbb{R}$ with $|t|<1$, where in each of ($\ref{eq-poly1}$) and ($\ref{eq-poly2}$) the constant $C>0$ may depend on $N$.

We now state and prove our lower bound for the relative perimeter.

\begin{lemma}
\label{lem1}
For each $\delta>0$ there exists $\epsilon>0$ such that for all  $A\in \Pi_+(\eta,r)$  and $E\subset \mathscr{C}_A$ as above with $|E|=\mathcal{H}^{N-1}(A)/N$ and having smooth boundary, if the function $u:A\rightarrow \mathbb{R}$ defined as in (\ref{eq-un}) satisfies
\begin{align}
\lVert u\rVert_{L^\infty(A)}<\epsilon\quad\textrm{and}\quad \lVert\nabla u\rVert_{L^\infty(A)}<\epsilon\label{hyp-1}
\end{align}
then
\begin{align*}
P(E;\mathscr{C}_{A})&\geq \mathcal{H}^{N-1}(A)\\
&\hspace{0.4in}+\frac{1-\delta}{2}\Bigl(\lVert \nabla u\rVert_{L^2(A)}^2-(N-1)\lVert u\rVert_{L^2(A)}^2\Bigr)-\frac{\delta}{2}\lVert u\rVert_{L^2(A)}^2.
\end{align*}
\end{lemma}
\begin{proof}
We argue as in \cite{F}.  Using ($\ref{eq-area}$) and ($\ref{eq-sqrt}$), write
\begin{align}
\nonumber P(E;\mathscr{C}_{A})&\geq \int_{A} (1+u)^{N-1}d\mathcal{H}^{N-1}\\
&\hspace{0.4in}+\frac{1}{2}\int_{A} (1+u)^{N-3}|\nabla u|^2\bigg(1-\frac{|\nabla u|^2}{4(1+u)^2}\bigg)d\mathcal{H}^{N-1}\label{A1}
\end{align}

Let $\delta>0$ be given, and fix $0<\epsilon<1$ to be chosen later in the argument.  Let $0<\delta_0<1$ be a parameter which will also be chosen later.  Suppose that $u$ (and thus $E$) is such that both inequalities in ($\ref{hyp-1}$) are satisfied.  We then have
\begin{align*}
\frac{|\nabla u|^2}{4(1+u)^2}&<\frac{\epsilon^2}{4(1-\epsilon)^2},
\end{align*}
which is bounded by $\delta_0$ for $\epsilon$ sufficiently small.  The right-hand side of ($\ref{A1}$) is thus bounded from below by 
\begin{align*}
&\int_{A} (1+u)^{N-1}+\frac{1-\delta_0}{2}(1+u)^{N-3}|\nabla u|^2d\mathcal{H}^{N-1}\\
&\hspace{0.2in}\geq \int_{A} (1+u)^{N-1}+\frac{1-\delta_0}{2}\Big(1-(N-3)\lVert u\rVert_{L^\infty(A)}\Big)|\nabla u|^2d\mathcal{H}^{N-1}
\end{align*}
where to obtain the last inequality we have used the elementary bound ($\ref{eq-tk}$) (with $k=N-3$).  Since we have assumed $\lVert u\rVert_{L^\infty}<\epsilon$, this last expression is larger than
\begin{align*}
\int_{A} (1+u)^{N-1}+\frac{(1-\delta_0)^2}{2}|\nabla u|^2d\mathcal{H}^{N-1}
\end{align*}
provided that $\epsilon$ is chosen sufficiently small to ensure $(N-3)\epsilon<\delta_0$.

We now estimate the first term from below. 
Using polar coordinates, we see that the condition $|E|=\mathcal{H}^{N-1}(A)/N$ 
is equivalent to
\begin{align}
\int_{A} 1\,d\mathcal{H}^{N-1}=\mathcal{H}^{N-1}(A)=\int_{A} (1+u)^Nd\mathcal{H}^{N-1}\label{eq-volcst}
\end{align}
therefore 
\begin{align*}
\int_{A} (1+u)^{N-1}d\mathcal{H}^{N-1}&=\int_{A} (1+u)^{N-1}d\mathcal{H}^{N-1}-\frac{N-1}{N}\int_{A} \Big((1+u)^{N}-1\Big)d\mathcal{H}^{N-1}\\
&\geq \int_A \Big(1-\frac{N-1}{2}u^2-C|u|^3\Big)d\mathcal{H}^{N-1}\\
&\geq \mathcal{H}^{N-1}(A)-(1+C\epsilon)\frac{N-1}{2}\int_A u^2d\mathcal{H}^{N-1},
\end{align*}
where to obtain the first inequality we have used the first inequality in ($\ref{eq-poly1}$), while to obtain the second inequality we have again used the hypothesis $|u|<\epsilon$.

Putting this bound together with the above estimates, we have shown 
\begin{align}
\nonumber &P(E;\mathscr{C}_{A})-\mathcal{H}^{N-1}(A)\\
\nonumber &\hspace{0.2in}\geq \frac{(1-\delta_0)^2}{2}\lVert \nabla u\rVert_{L^2(A)}^2 -(1+C\epsilon)\frac{N-1}{2}\lVert u\rVert_{L^2(A)}^2\\
&\hspace{0.2in}=\frac{(1-\delta_0)^2}{2}\Big(\lVert \nabla u\rVert_{L^2(A)}^2 -(N-1)\lVert u\rVert_{L^2(A)}^2\Big)-C_*(\epsilon,\delta_0)\lVert u\rVert_{L^2(A)}^2\label{A3}.
\end{align}
with
\begin{align*}
C_*(\epsilon,\delta_0)=\frac{(C\epsilon-\delta_0^2+2\delta_0)(N-1)}{2}.
\end{align*}

Choosing $\epsilon$ sufficiently small so that $C\epsilon<\delta_0^2$, and then choosing $\delta_0$ sufficiently small so that $(1-\delta_0)^2>1-\delta$ and $(N-1)\delta_0<\delta$, this completes the proof of the lemma.
\end{proof}

With the proof of Lemma $\ref{lem1}$ complete, we are ready to establish Theorem $\ref{clm11}$.  The proof is an argument by contradiction: supposing that the desired characterization of isoperimetric sets fails, we choose a sequence of minimizing sets which exhibits this, and then apply compactness results for almost minimizers of perimeter together with: (i) the characterization obtained in \cite{LP} (see also \cite{M} and \cite{FI}) of isoperimetric sets in the convex cones $\mathscr{C}_{A_*}$, (ii) the regularity theory of almost minimal surfaces, (iii) the lower bound of Lemma $\ref{lem1}$, and (iv) the Poincar\'e inequality given in Corollary $\ref{cor-poincare}$.  This procedure allows us to show that the minimizing sets, being close to sections of balls, are in fact such sections themselves (this is the outcome of a ``stability'' type inequality in the spirit of Fuglede's inequality for nearly-spherical sets \cite{F}), giving the desired contradiction.

\begin{proof}[Proof of Theorem $\ref{clm11}$]
Let $r>0$, and suppose the claim fails.  Then there exist sequences $(A_{*,n})$, $(A_n)\subset\Pi_+(\eta,r)$ such that each $\mathscr{C}_{A_{*,n}}$ is convex, 
\begin{align}
d_{L^\infty}(\partial A_n,\partial A_{*,n})\rightarrow 0\label{eq-plus-1}
\end{align}
as $n\rightarrow\infty$, and for each $n\geq 1$ there exists a value $m_n>0$ and a set of finite perimeter $E_n\subset \mathscr{C}_{A_n}$ with 
\begin{enumerate}
\item[(i)] $|E_n|=m_n$ for each $n\geq 1$,
\item[(ii)] each $E_n$ minimizes $P(\,\cdot\,;\mathscr{C}_{A_n})$ with respect to the constraint $|\cdot|=m_n$, 
\item[(iii)] $E_n\neq B_{m_n;A_n}\cap \mathscr{C}_{A_n}$,
\end{enumerate}
where for each $m\geq 0$ and $A\subset S_+$, $B_{m;A}$ denotes the ball centered at the origin with radius chosen so that $|B\cap \mathscr{C}_A|=m$.

Rescale each set $E_n$ so that
\begin{align*}
|E_n|=\frac{\mathcal{H}^{N-1}(A_n)}{N}. 
\end{align*}
The sequence of rescaled sets $(E_n)$ then retains the properties (i)--(iii), with $m_n=\mathcal{H}^{N-1}(A_n)/N$ for $n\geq 1$.  Moreover, the sets $E_n$ are almost minimizers for the perimeter functional in the cones $\mathscr{C}_{A_n}$ 
in the sense of Almgren \cite{Almgren}; this notion corresponds to the presence of the volume constraint, and asserts that there exist $\Lambda \in [0,\infty)$ and $r>0$ such that 
\begin{align*}
P(E_n;B\cap \mathscr{C}_{A_n})\leq P(F;B\cap \mathscr{C}_{A_n})+\Lambda|E_n\Delta F|,
\end{align*}  
whenever $B$ is a ball of radius at most $r$ and the symmetric difference $E_n\Delta F$ is  contained in $B\cap \mathscr{C}_{A_n}$.  We refer to \cite[Part III]{Maggi} for precise definitions and a comprehensive treatment of the relevant results; see also \cite{Giusti} for another treatment of the theory.

We next invoke compactness results for convex sets and almost minimizers of perimeter.  Concerning the convex sets $\mathscr{C}_{A_{*,n}}$, we can find $A_*\in \Pi_+(\eta,r)$ with $\mathscr{C}_{A_*}$ convex such that, after extracting a subsequence,
\begin{align*}
d_{L^\infty}(\partial A_{*,n},\partial A_*)\rightarrow 0
\end{align*}
as $n\rightarrow\infty$; in view of ($\ref{eq-plus-1}$), this also gives $d_{L^\infty}(\partial A_n,\partial A_*)\rightarrow 0$ as $n\rightarrow\infty$.  Turning to the sets $E_n$, 
the argument in \cite{RR} shows that  $E_n$ cannot lose mass nor perimeter at infinity.
Hence, \cite[Proposition 21.13 and Theorem 21.14]{Maggi} lead to existence of a set of finite perimeter $E_*\subset \mathbb{R}^N$ such that, after again extracting a subsequence, 
\begin{align*}
|E_n\Delta E_*|\rightarrow 0
\end{align*}
as $n\rightarrow\infty$ and
\begin{align*}
P(E_*;\mathscr{C}_{A_*})\leq \liminf_{n\rightarrow\infty} P(E_n;\mathscr{C}_{A_n}).
\end{align*}

Standard arguments now show that $|E_*|=m_*:=\mathcal{H}^{N-1}(A_*)/N$ and that $E_*$ is a minimizer for $P(\, \cdot\, ;\mathscr{C}_{A_*})$ subject to the constraint $|\cdot|=m_*$.  It therefore follows from the result of \cite{LP} (see also \cite{M,FI} for related results) that $E_*=B_{m_*;A_*}\cap A_*$.
Then, it follows by the regularity theory for almost minimizers both in the interior (see for instance \cite[Part III]{M}) and up to the boundary (see for instance \cite{DM}) that $\partial E_n$
can be written as a $C^1$ graph over $A_n$.

To conserve notation, we let $(A_n)$ and $(E_n)$ denote the subsequences chosen in the compactness argument above, and fix $n\geq 1$.  
For each $\xi\in A_n$, write 
\begin{align*}
u_n(\xi)=\sup\, \Big\{t>-1:(1+t)\xi\in E_n\Big\},
\end{align*}
so that $\partial E_n\cap \mathscr{C}_{A_n}$ has the parametrization
\begin{align*}
(1+u_n(\xi))\xi,\quad \xi\in A_n.
\end{align*}
A first variation argument gives $(\nabla u_n)\cdot \nu=0$ on $\partial A_n$ (here, for each $x\in \partial A_n$, $\nu=\nu(x)$ is the unit normal to $\partial \mathscr{C}_{A_n}$ at $x$; this boundary condition corresponds to the fact that since we are minimizing the relative perimeter, the minimal surface attaches to $\partial \mathscr{C}_{A_n}$ orthogonally).

Observe that the convergence of $E_n$ to $E_*$ (which is a sector of a ball in $\mathbb{R}^N$) and the almost minimality of the sets $E_n$ with respect to the perimeter functional imply $\lVert u_n\rVert_{L^\infty(A_n)}\rightarrow 0$ and $\lVert \nabla u_n\rVert_{L^\infty(A_n)}\rightarrow 0$ as $n\rightarrow\infty$.  Fixing $\delta>0$, we may therefore apply Lemma $\ref{lem1}$ to obtain
\begin{align}
\nonumber P(E_n;\mathscr{C}_{A_n})&\geq \mathcal{H}^{N-1}(A_n)\\
&\hspace{0.2in}+\frac{1-\delta}{2}\Bigl(\lVert \nabla u_n\rVert_{L^2(A_n)}^2-(N-1)\lVert u_n\rVert_{L^2(A_n)}^2\Bigr)-\frac{\delta}{2}\lVert u_n\rVert_{L^2(A_n)}^2\label{eq-lowerbd-n}
\end{align}
for $n$ sufficiently large (depending on $\delta$).

We now prepare to apply the Poincar\'e inequality of Corollary $\ref{cor-poincare}$.  For this, we require some bounds on $\overline{u}_n:=\frac{1}{\mathcal{H}^{N-1}(A_n)}\int_{A_n} u_nd\mathcal{H}^{N-1}$.  To obtain these, we fix $\epsilon>0$ and appeal to ($\ref{eq-volcst}$) as in the proof of Lemma $\ref{lem1}$, which gives
\begin{align*}
|\overline{u}_n|&=\frac{1}{\mathcal{H}^{N-1}(A_n)}\bigg|\int_{A_n} \bigg(u_n-\frac{1}{N}\Big((1+u_n)^N-1\Big)\bigg)d\mathcal{H}^{N-1}\bigg|\\
&\leq \frac{C}{\mathcal{H}^{N-1}(A_n)}\int_{A_n} |u_n|^2d\mathcal{H}^{N-1}\\
&\leq \frac{C\epsilon}{\mathcal{H}^{N-1}(A_n)}\int_{A_n} |u_n|\,d\mathcal{H}^{N-1},
\end{align*}
where in last inequality we have recalled that $|u_n|<\epsilon$ for $n$ sufficiently large.

An application of the triangle inequality followed by H\"older's inequality now leads to the bound
\begin{align*}
|\overline{u}_n|&\leq \frac{C\epsilon}{\mathcal{H}^{N-1}(A_n)}\int_{A_n} |u_n-\overline{u}_n|\,d\mathcal{H}^{N-1}+C\epsilon|\overline{u}_n|\\
&\leq \frac{C\epsilon}{\mathcal{H}^{N-1}(A_n)^{1/2}}\left(\int_{A_n} |u_n-\overline{u}_n|^2d\mathcal{H}^{N-1}\right)^{1/2}+C\epsilon|\overline{u}_n|.
\end{align*}
Hence, for $\epsilon<1/C$ (where $C>0$ is a dimensional constant), we obtain
\begin{align*}
|\overline{u}_n|&\leq \frac{C\epsilon}{(1-C\epsilon)\mathcal{H}^{N-1}(A_n)^{1/2}}\lVert u_n-\overline{u}_n\rVert_{L^2(A_n)}
\end{align*}
for all $n\geq n_0(\epsilon)$, with $n_0(\epsilon)$ sufficiently large.

With this bound on $\overline{u}_n$ in hand, we are now ready to apply Corollary $\ref{cor-poincare}$.  Returning to ($\ref{eq-lowerbd-n}$), we obtain
\begin{align*}
&P(E_n;\mathscr{C}_{A_n})-\mathcal{H}^{N-1}(A_n)\\
&\hspace{0.2in}\geq \frac{1-\delta}{2}\lVert \nabla u_n\rVert_{L^2(A_n)}^2-\bigg(\frac{(N-1)(1-\delta)}{2}+\frac{\delta}{2}\bigg)\\
&\hspace{2.2in} \cdot\bigg(\lVert u_n-\overline{u}_n\rVert_{L^2(A_n)}+|\overline{u}_n|\mathcal{H}^{N-1}(A_n)^{1/2}\bigg)^2\\
&\hspace{0.2in}\geq \frac{1-\delta}{2}\lVert \nabla u_n\rVert_{L^2(A_n)}^2-\bigg(\frac{(N-1)(1-\delta)}{2}+\frac{\delta}{2}\bigg)\\
&\hspace{2.2in} \cdot\bigg(1+\frac{C\epsilon}{1-C\epsilon}\bigg)^2\lVert u_n-\overline{u}_n\rVert_{L^2(A_n)}^2,
\end{align*}
so that application of the estimate in Corollary $\ref{cor-poincare}$ gives
\begin{align*}
&P(E_n;\mathscr{C}_{A_n})-\mathcal{H}^{N-1}(A_n)\geq C(N,\eta,\epsilon,\delta)\lVert \nabla u_n\rVert_{L^2(A_n)}^2
\end{align*}
for $n$ sufficiently large (depending on $\epsilon$ and $\delta$), with
\begin{align*}
C(N,\eta,\epsilon,\delta):=\frac{1-\delta}{2}-\bigg(\frac{(N-1)(1-\delta)}{2}+\frac{\delta}{2}\bigg)\bigg(1+\frac{C\epsilon}{1-C\epsilon}\bigg)^2\frac{1}{c_1(\eta)^2}.
\end{align*}

On the other hand, since $E_n$ minimizes $E\mapsto P(E;\mathscr{C}_{A_n})$ among sets of finite perimeter with $|E|=\mathcal{H}^{N-1}(A_n)/N$, we have
\begin{align*}
P(E_n;\mathscr{C}_{A_n})-\mathcal{H}^{N-1}(A_n)\leq P(B;\mathscr{C}_{A_n})-\mathcal{H}^{N-1}(A_n)=0
\end{align*}
where $B\subset\mathbb{R}^N$ is the unit ball centered at the origin.

Since
\begin{align*}
\lim_{\epsilon,\delta\rightarrow 0} C(N,\eta,\epsilon,\delta)=\frac{1}{2}-\frac{N-1}{2c_1(\eta)^2}>0
\end{align*}
we may choose $\epsilon_0>0$ and $\delta_0>0$ such that $C(N,\eta,\epsilon_0,\delta_0)>0$.  It follows that there exists $n_1=n_1(\epsilon_0,\delta_0)$ such that for all $n\geq n_1$ we have
\begin{align*}
0\leq C(N,\eta,\epsilon_0,\eta_0)\lVert \nabla u_n\rVert_{L^2(A_n)}^2\leq 0,
\end{align*}
giving $\lVert \nabla u_n\rVert_{L^2(A_n)}^2=0$.  This implies that $u_n$ is constant, so that the volume constraint ensures $u_n\equiv 1$ and thus $E_n=B\cap \mathscr{C}_{A_n}$, contradicting our initial choice of the sets $E_n$ (see condition (iii) at the beginning of the argument).  This completes the proof of the theorem.
\end{proof}

We conclude by pointing out several issues of interest in connection with the above arguments.

\begin{enumerate}
\item In view of Theorem $\ref{clm11}$, it is natural to ask if the result can be extended to cones $\mathscr{C}_A$ with $A\subset S_{>0}:=\{\xi\in S^{N-1}:\pi_N(\xi)>0\}$.  In this case, the constant $c_1$ in the relevant Poincar\'e inequality is no longer necessarily strictly greater than $\sqrt{N-1}$, and a finer analysis would be required.  

Nevertheless, it is worth pointing out that it may be possible for the compactness argument presented here to be adapted to give some information: directly expanding $u_n=\sum_{k=0}^\infty c_{n,k}Z_k^{(n)}$ as a sum of $A_n$--eigenfunctions in the spirit of the remark at the end of Section 2.3, the key issue becomes to obtain (uniform in $n$) smallness of $|c_{0,n}|$ and $|c_{1,n}|$ in comparison with $\lVert u_n\rVert_{L^2}$.  In the argument of \cite{F} treating the unconstrained case, i.e., $A=S^{N-1}$, such control is obtained as a consequence of the translation invariance of the problem (fixing the barycenter at the origin).
\item In \cite{LP}, the authors give explicit counterexamples of non-convex domains for which the volume-constrained minimizer is not a section of a ball.  These examples take the form of the union of two (relatively large) disjoint cones joined by a thin ``neck.''  In view of our results, one might ask if there is a measure-theoretic characterization of cones for which the implication ``isoperimetric minimizer implies section of a ball'' holds.  

In a related matter, observing that the counterexample coincides with the classical counterexample of Courant and Hilbert for stability of the Neumann eigenvalues under domain perturbations (see \cite{Courant}, \cite{HempelSecoSimon}), it would be interesting to further investigate possible connections between these related topics.
\end{enumerate}

\end{document}